\documentclass[a4paper,11pt,leqno]{amsart}
\usepackage{cite}
\usepackage[english]{babel}
\usepackage[utf8]{inputenc}
\usepackage[T1]{fontenc}
\usepackage{amsthm}
\usepackage{amssymb}
\usepackage{amsmath}

\usepackage{amsfonts}
\usepackage{graphicx}
\usepackage[all]{xy}
\newtheorem{Thm}{Theorem}
\newtheorem{Prop}[Thm]{Proposition}
\newtheorem{Cor}[Thm]{Corollary}
\newtheorem{Apulause}[Thm]{Lemma}

\theoremstyle{definition}

\newtheorem{Not}[Thm]{Remark}

\renewcommand{\r}{|}
\newcommand{\R}{\mathbb{R}}
\newcommand{\C}{\mathcal{C}}

\renewcommand{\S}{\mathcal{S}}
\newcommand{\G}{\mathcal{G}}
\newcommand{\D}{\mathcal{D}}

\newcommand{\es}{\varnothing} 

\renewcommand{\Cup}{\bigcup}

\newcommand{\ra}{\rightarrow}

\newcommand{\al}{\alpha}
\newcommand{\be}{\beta}

\title{Note on the canonical genus of a knot}
\author{Martina Aaltonen}
\address{Department of Mathematics and Statistics, P.O.Box 68, 00014 Univeristy of Helsinki}
\email{martina.aaltonen@helsinki.fi}
\date{\today}  
 
\begin{document}
\begin{abstract}We show that every canonical Seifert surface is (up to isotopy) given by a knot diagram in which the (open) Seifert disks are pairwise disjoint.
\end{abstract}
\maketitle
\section{Introduction}

In this article we consider canonical Seifert surfaces and the canonical genus of PL knots \cite{SEI}. 
We prove the following:

\begin{Thm}\label{tulos}Let $M \subset \R^3$ be a Seifert surface that is canonical with respect to a knot diagram $\D.$ Then there exists a knot diagram $\D'$ in which the (open) Seifert disks are pairwise disjoint and a Seifert surface that is canonical with respect to $D'$ and isotopic to $M.$
\end{Thm}

Recently, it was pointed out to the author that this Theorem
is already known, see \cite{HIR}. Theorem \ref{tulos} imposes the classical result that the fundamental group of $\R^3 \setminus M$ is a free group for any canonical Seifert surface $M$ of a knot. Theorem \ref{tulos} also has the following corollary. 

\begin{Cor}\label{pu} The canonical genus of a knot is achieved in a subclass of knot diagrams for which the (open) Seifert disks are pairwise disjoint.
\end{Cor}

\begin{figure}[h!]
\includegraphics{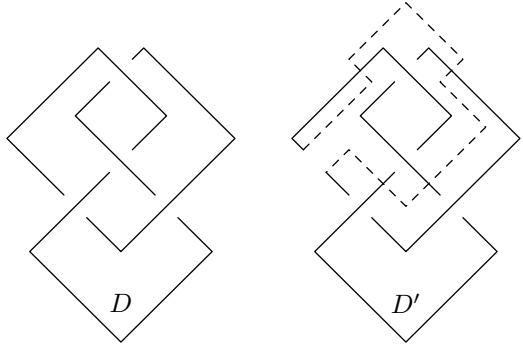}\caption{In the knot diagram $\D'$ the evenly dotted PL curve lies on top of the black PL curve coming from the knot diagram $\D.$ The diagrams $D$ and $D'$ are both knot diagrams for the figure eight knot satisfying $\G(\D)=\G(\D')=1.$ The collection of (open) Seifert disks in $\D$ is not pairwise disjoint, but the collection of (open) Seifert disks in $\D'$ is pairwise disjoint.}
\label{werry}
\end{figure}

In \cite[Ch. VII]{K} Kauffman has shown how to construct from a knot diagram a link diagram in which the open Seifert disks are pairwise disjoint by adding simple closed polygons on top of the knot diagram. Based on this observation we give an algorithm to construct from any knot diagram such a knot diagram in which the open Seifert disks are pairwise disjoint. We show that this modification corresponds to an isotopy between two Seifert surfaces that are canonical with the respective knot diagrams.\\ 

\textbf{Acknowledgements.}
I want to thank Sebastian Baader for his helpfulness and his excellent suggestions, Pekka Pankka for reading the manuscript and his suggestions for the presentation and Vadim Kulikov for discussions on the topic.   

\section{Notation and preliminaries}

We recall that PL knots are isotopy classes of polyhedral embeddings $k \colon S^1 \ra \R^3.$ For the convenience of presentation we denote $\R^2=\R^2 \times \{0\}$ throughout this article. Let $d \colon S^1 \ra \R^2$ be a polyhedral curve so that 
\begin{itemize} 
\item[(a)] $\mathrm{Im}(d)\subset \R^2$ is a polygon with vertices $v_1, \ldots, v_k \in \R^2$ satisfying $\#(d^{-1}\{v_j\})=1$ for every $j \in \{0, \ldots, k\},$
\item[(b)] $\#(d^{-1}\{z\}) \leq 2$ for all $z \in \R^2$ and
\item[(c)] the set $\C(d):=\{z \in \R^2  \mid \#(d^{-1}\{z\})=2\}$ is finite.
\end{itemize}
Let $f \colon d^{-1}(\C(d)) \ra \{0,1\}$ be a function satisfying $f(d^{-1}\{z\})=\{0,1\}$ for all $z \in \C(d).$ We call the pair $\D:=(d,f)$ a \textit{knot diagram} and $\C(\D):=\C(d)$ the \textit{crossings} of $\D.$ In what follows we fix an orientation of $S^1$ and note by $[x_1,x_2] \subset S^1$ the segment from $x_1$ to $x_2$ in this orientation. 

Let $k \colon S^1 \ra \R^3$ be a polyhedral embedding, $p \colon \R^3 \ra \R^2$ the projection and $\D=(d,f)$ a knot diagram. We say that $p$ \textit{induces} $\D$ \textit{from} $k$ if $d=p \circ k$ and for $x_1 \in f^{-1}\{1\}$ and $x_2 \in f^{-1}\{0\}$ satisfying $d(x_1)=d(x_2)$ the third coordinate of $k(x_1)$ is bigger than the third coordinate of $k(x_2).$  

Let $\D:=(d,f)$ be a knot diagram. A restriction map 
$s:=d \r E$ is a \textit{Seifert circuit} if $E=\Cup_{i=0}^k [x_i,y_i]$ for such points $x_0, \ldots, x_k$ and $y_0, \ldots, y_k$ in $d^{-1}(\C(\D))$ that
\begin{itemize}
\item[(a)] $[x_i,y_i] \cap d^{-1}(\C(\D)) = \{x_i,y_i\}$ for $i \in \{0, \ldots, k\}$ and  
\item[(b)] $d(x_i)=d(y_{i-1})$ and $f(x_i)\neq f(y_{i-1})$ for all $i \in \{1, \ldots, k\}$ and  $d(x_0)=d(y_k)$ and $f(x_0)\neq f(y_k).$  
\end{itemize}
We denote by $\S(\D)$ the collection of Seiferts circuits in $\D$ and for every $s \in \S(\D)$ we denote $\C(s):=\C(\D) \cap \mathrm{Im}(s).$ 

Let $\D=(d,f)$ be a knot diagram. For a Seifert circuit $s \in \S(\D)$ the image $\mathrm{Im}(s) \subset \R^2$ is a simple closed polygon. We call an open bounded connected subset $b_s \subset \R^2$ satisfying $\partial b_s=\mathrm{Im}(s)$ the (open) \textit{Seifert disk} bounded by $s.$ 
Now for every point $z \in \mathrm{Im}(d)$ there exists $s \in \S(\D)$ so that $z \in \mathrm{Im}(s)$ and for every $z \in \C(\D)$ there exists a unique pair $s,t \in \S(\D)$ so that $z \in \mathrm{Im}(s) \cap \mathrm{Im}(t).$

We denote
$$\S_{II}(\D):=\{s \in \S(\D) \mid b_{t} \subsetneqq b_s \text{ for some }t \in \S(\D)\}$$
and for every $s \in \S(\D)$ we denote
$$\C_{II}(s):=\{z \in \C(s) \mid  z \in \C(t) \text{ for } t \in \S(\D) \text{ satisfying } b_{t} \subsetneqq b_s \}.$$

\begin{Apulause}\label{1} Let $\D$ be a knot diagram. Then the open Seifert disks are pairwise disjoint if and only if $\S_{II}(\D)=\es.$
\end{Apulause}
\begin{proof} Suppose that $\S_{II}(\D)= \es,$ but that there exists Seifert circuits $s, t \in \S(\D), s\neq t,$ so that $b_s \cap b_t \neq \es.$ Then there exists a point in $\mathrm{Im}(t) \cap b_s.$ By property (a) of a knot diagram and property (b) of a Seifert circuit the simple closed polygons $\mathrm{Im}(t)$ and $\mathrm{Im}(s)$ do not cross each other. Thus $\mathrm{Im}(t) \subset \mathrm{Im}(s) \cup b_s.$ In particular, $b_t \subset b_s$ and $b_s \cap b_t =b_t.$ Thus $s \in \S_{II}(\D),$ which is a contradiction. Thus the open Seifert disks are pairwise disjoint. Suppose $\S_{II}(\D)\neq \es$ and $s \in \S_{II}(\D).$ Then there exists trivially a pair of Seifert disks with non-empty intersection. 
\end{proof}

For the definition of a canonical Seifert surface, we give first the definition of a crossing area. Let $\D=(d,f)$ be a knot diagram and $z \in \C(\D).$ Let $z_1, z_2 \in \mathrm{Im}(s)$ be points so that $[z,z_{1}], [z,z_{2}] \subset \mathrm{Im}(s)$ and 
$[z_{1},z_{2}] \cap \mathrm{Im}(d)=\{z_{1}, z_{2}\}$ and let $\Delta^z_{i} \subset \R^2$ be the closed set that has the polygon $[z,z_{1}] \cup [z,z_{2}] \cup [z_{1},z_{2}]$ as its boundary for every $s \in \S(\D)$ satisfying $z \in \C(s).$ Then we call $w_z:=\Cup_{i \in \{s,t\}} \Delta^z_{i},$ where $s,t \in \S(\D)$ are distinct Seiferts circles satisfying $z \in \C(s) \cap C(t),$ a \textit{crossing area} of $z.$ 

\begin{Not} Let $\D$ be a knot diagram $s \in \S(\D)$ and $(w_i)_{i \in \C(s)}$ be a collection of pairwise disjoint crossing areas. Let $b_s \subset \R^2$ be the open Seifert disk bounded by $s.$ Then the boundary of $b_s \setminus \big(\Cup_{i \in \C(s)} \Delta^i_s \big) \subset \R^2$ is a simple closed polygon. Further,
\vspace{0,15cm} 
\begin{center}
$b_s \setminus \big(\Cup_{i \in \C(s)} \Delta^i_s \big)=b_s \setminus \big(\Cup_{i \in \C(s) \setminus \C_{II}(s)} \Delta^i_s \big)=b_s \setminus \big(\Cup_{i \in \C(s)\setminus \C_{II}(s)} w_i \big),$
\end{center}
\vspace{0,15cm} 
since $b_s \cap \Delta^i_s=\es$ for $i \in \C_{II}(s)$ and $b_s \cap w_i=b_s \cap \Delta^i_s$ for $i \in \C(s) \setminus \C_{II}(s).$ 
\end{Not} 


Let $M \subset \R^3$ be a Seifert surface and $\D=(d,f)$ a knot diagram. Then $M$ is \textit{canonical with respect} to $\D,$ if $\D$ is induced from $\mathrm{Bd}(M)$ by the projection $p \colon \R^3 \ra \R^2$ and there exists a triple $((u_i)_{i \in \S(\D) \cup \C(\D)},(w_i)_{i \in \C(\D)},(r_i)_{i \in \S(\D)}),$ where $(u_i)_{i \in \S(\D) \cup \C(\D)}$ is a collection of PL $2$-manifolds in $\R^3$ with one boundary component and genus $0,$ $(w_i)_{i \in \C(\D)}$ is a collection of pairwise disjoint crossing areas of $\D$ and $(r_i)_{i \in \S(\D)}$ is a collection of real numbers, so that 
\begin{itemize}
\item[(a)]$M=\Cup_{i \in \S(\D) \cup \C(\D)} u_i,$ 
\item[(b)]if $i,j \in  \S(\D) \cup \C(\D)$ satisfy $i \in \S(\D)$ and $j \in \C(i)$ or $j \in \S(D)$ and $i \in \C(j),$ then $u_i \cap u_j= \mathrm{Bd}(u_i) \cap \mathrm{Bd}(u_j) \approx [0,1],$ otherwise $u_i \cap u_j=\es,$ 
\item[(c)]$u_i=\big\{z+(0,0,r_i) \mid z \in \mathrm{Cl}\big(b_i \setminus \Cup_{j \in \C(i)\setminus \C_{II}(i)}w_j\big)\big\}$ for $i \in \S(\D)$ and $b_i \subset \R^2$ the open Seifert disk bounded by $i,$ 
\item[(d)]$u_i \subset p^{-1}(w_i)$ for $i \in \C(\D)$ and 
\item[(e)]$u_j \cap (\R^2 +(0,0,r_i))=u_j \cap u_i$ for $i \in \S(\D)$ and $j \in \C_{II}(i).$
\end{itemize}
We say that the triple $((u_i)_{i \in \S(\D) \cup \C(\D)},(w_i)_{i \in \C(\D)},(r_i)_{i \in \S(\D)})$ is the data of a canonical Seifert surface $M$ with respect to $\D.$

We say that a Seifert surface $M$ of a knot is a \textit{canonical Seifert surface}, if $M$ is isotopic to a Seifert surface that is canonical with respect to a knot diagram $\D$ of that knot. Note that, if a Seifert surface $M$ is canonical with respect to $\D,$ then the genus $\G(\D)$ of $M$ is $(\#\C(\D)-\#\S(\D)+1)/2.$ 

We recall that using Seiferts algorithm \cite{ROL} one can construct for every knot diagram $\D$ a Seifert surface $M \subset \R^3$ that is canonical with respect to $\D.$ Indeed, let $(w_i)_{i \in \C(\D)}$ be a collection of pairwise disjoint crossing areas of $\D$ and $(r_i)_{i \in \S(\D)}$ a collection of real numbers so that by defining 
\vspace{0,15cm}
\begin{center}
 $u_i:=\big\{z + (0,0,r_i) \mid z \in \mathrm{Cl}\big(b_i \setminus \Cup_{j \in \C(i)\setminus \C_{II}(i)} w_j\big)\big\}$
\end{center}
\vspace{0,15cm} 
for every $i \in \S(\D)$ the collection $(u_i)_{i \in \S(\D)}$ is pairwise disjoint and 
$$\{z+(0,0,r) \mid z \in w_k, r \in [r_{i},r_{j}]\} \cap \{u_l \mid l \in \S(\D)\} \subset u_i \cap u_j$$ for every $i,j \in \S(\D)$ and $k \in \C(i) \cap \C(j).$ Then there exists a collection $(u_i)_{i \in \C(\D)}$ of bands corresponding to crossings, see Figures \ref{krrr} and \ref{wer}, so that 
$M:=\Cup_{i \in \S(\D) \cup \C(\D)} u_i$ is a Seifert surface and the data of $M$ with respect to $\D$ is the triple $((u_i)_{i \in \S(\D) \cup \C(\D)},(w_i)_{i \in \C(\D)},(r_i)_{i \in \S(\D)}).$

\begin{figure}[htb]
\includegraphics{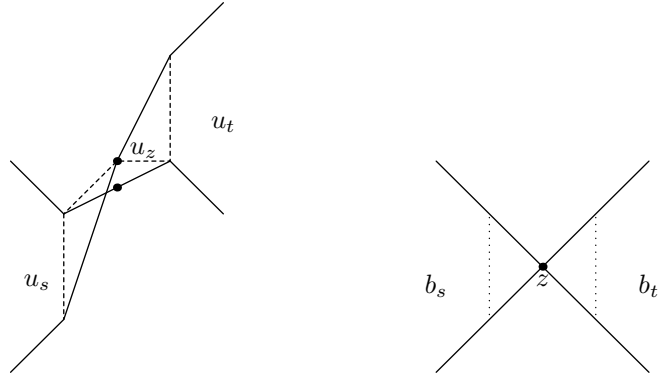}
\caption{For $\D=(d,f),$ $s,t \in \S(\D)$ and $z \in (\C(s)\setminus \C_{II}(s)) \cap \C(t)$ the band $u_z$ connecting $u_{s}$ and $u_{t}$ on the left and $\mathrm{Im}(d) \subset \R^2$ around the crossing $z$ on the right.}\label{krrr}
\end{figure} 
 
\begin{figure}[htb]
\includegraphics{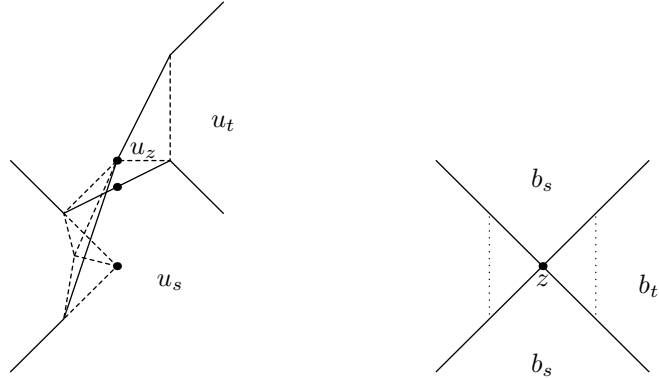}
\caption{For $\D=(d,f),$ $s,t \in \S(\D)$ and $z \in \C_{II}(s) \cap \C(t)$ the band $u_z$ connecting $u_{s}$ and $u_{t}$ on the left and $\mathrm{Im}(d) \subset \R^2$ around the crossing $z$ on the right.}\label{wer}
\end{figure}

\section{Proof of Theorem \ref{tulos}}


\begin{figure}[htb]
\includegraphics{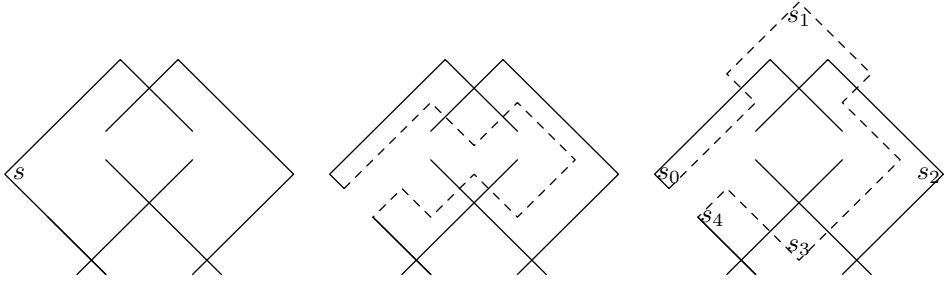}
\caption{On the left an illustration of $s \in \S_{II}(\D).$ In the middle is illustrated the changes in the induced knot diagram obtained by the first isotopy that takes $\be$ to $\be_1.$ On the right is illustrated the changes in the induced knot diagram by the second isotopy that takes $\be_1$ to such $\be_2$ that $p(\be_2) \cap \mathrm{Im}(d) \subset \mathrm{Im}(s)$ and the Seifert circuits $s_0, \ldots, s_4 \in \S(D_2)\setminus \S_{II}(D_2).$}
\label{werryt}
\end{figure}

\begin{Prop}\label{3} Let $\D=(d,f)$ be a knot diagram and $M \subset \R^3$ a Seifert surface that is canonical with respect to $\D.$ Suppose $\S_{II}(\D) \neq \es$ and $s \in \S_{II}(\D).$ Then there exists a knot diagram $\D'=(d',f')$ and a Seifert surface $M'$ so that 
\begin{itemize}
\item[(i)]$M'$ is canonical with respect to $\D',$ 
\item[(ii)]$M'$ is isotopic to $M,$ 
\item[(iii)]$\S_{II}(\D')=\S_{II}(\D) \setminus \{s\}$ and 
\item[(iv)]$\#\C(\D')-\#\C(\D)=2\#\C_{II}(s).$ 
\end{itemize}
\end{Prop} 
\begin{proof} Let $((u_i)_{i \in \S(\D) \cup \C(\D)},(w_i)_{i \in \C(\D)},(r_i)_{i \in \S(\D)})$ be the data of $M$ with respect to $\D.$ Let $s \in \S_{II}(\D).$ Then $\mathrm{Bd}(u_s) \subset (\R^3 +(0,0,r_s))$ is a simple closed polygon and there exists $\be=[z_2,z_1]\subset \mathrm{Bd}(M)$ so that $\be \subset \mathrm{Bd}(u_s) \setminus p^{-1}\big(\Cup_{i \in \C(\D)}w_i\big).$
Now $\al=(\mathrm{Bd}(u_s) \setminus [z_2,z_1]) \cup \{z_1,z_2\}$ is a simple PL curve and there exists an embedding $\iota \colon \al \ra u_s \setminus \mathrm{Bd}(u_s)$ so that the PL curve $\be_1=[z_1, \iota(z_1)] \cup \iota(\al) \cup [\iota(z_2),z_2]$ is such a simple curve in $(\R^2 +(0,0,r_s))$ that $F=\Cup_{z \in \al}[z,\iota(z)]$ is a PL disk having $\al \cup \be_1$ as its boundary. 
In particular, there exists an isotopy that takes $M$ to $M_1:=(M \setminus u_s) \cup F$ and keeps   
$M \setminus u_s$ fixed. Now the projection $p \colon \R^3 \ra \R^2$ induces a knot diagram $D_1:=(d_1,f_1)$ of $\mathrm{Bd}(M_1)$ so that 
$$\mathrm{Im}(d_1)=\big(\mathrm{Im}(d)\setminus p(\be)\big) \cup p(\be_1),$$
see Figures \ref{werryt} and \ref{kr}.  
\begin{figure}[ht]
\includegraphics{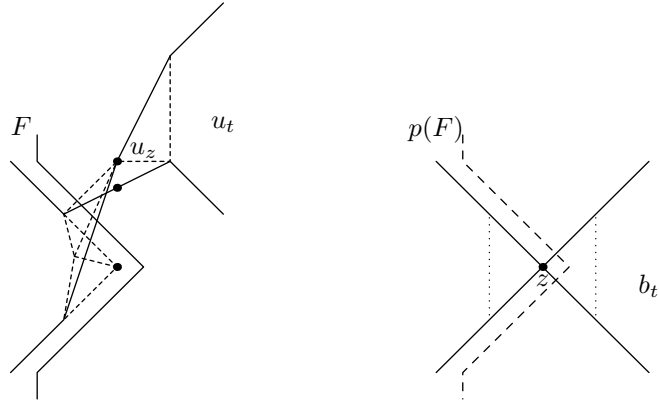}
\caption{For $z \in \C_{II}(s) \cap \C(t),$ $t \in \S(\D),$  
the band $u_z$ connecting $F$ and $u_{t}$ on the left and $\mathrm{Im}(d_1) \subset \R^2$ around the crossing $z$ on the right.}
\label{kr}
\end{figure} 

We note that $M_1$ is not in general canonical with respect to $D_1.$ 
However, by contracting $F$ linearly along the segments $[z,\iota(z)]$ for each $z \in \al$ towards $\al$ and rotating it for every $i \in \C_{II}(s)$ locally around $\al$ in a neighbourhood of $p^{-1}(w_i) \cap \al$ either upwards or downwards to avoid $u_i$ and then stretching it away from $\al$ in $(\R^2 +(0,0,r_s)),$ we have an isotopy that fixes $M_1 \setminus F$ and takes $F$ to a PL 2-manifold $G \subset \R^3$ and $\be_1$ to $\be_2$ so that $M_2:=(M_1 \setminus F) \cup G$ satisfies:
\begin{itemize}
\item[(a)]$\partial_{\R^2}(p(G))=p(\al) \cup p(\be_2)$ and $p(G)\cap \mathrm{Im}(d)=p(\al),$ 
\item[(b)] the projection $p \colon \R^3 \ra \R^2$ induces a knot diagram $\D_2:=(d_2,f_2)$ of $\mathrm{Bd}(M_2)$ so that 
$$\mathrm{Im}(d_2)=\big(\mathrm{Im}(d)\setminus p(\be)\big) \cup p(\be_2),$$ $$\C(\D_2)\setminus \C(\D)=p(\al \setminus \{z_1,z_2\}) \cap p(\be_2)$$ and 
$$\#(\C(\D_2)\setminus \C(\D))=2 \#\C_{II}(s)$$ 
\item[(c)]$w_i \cap p(\be_2)=\es$ for every $i \in \C(\D)$ and 
\item[(d)]there exists a collection $(w_i)_{i \in \C(\D_2) \setminus \C(\D)}$ of pairwise disjoint crossing areas in $\R^2 \setminus \left(\Cup_{i \in \C(\D)} w_i \right)$ so that 
$$G \setminus p^{-1}\{z \in w_i \mid i \in \C(\D_2) \setminus \C(\D)\}\subset (\R^2 + (0,0,r_{s})),$$
\end{itemize}
see Figures \ref{werryt} and \ref{krr}.

\begin{figure}[ht]
\includegraphics{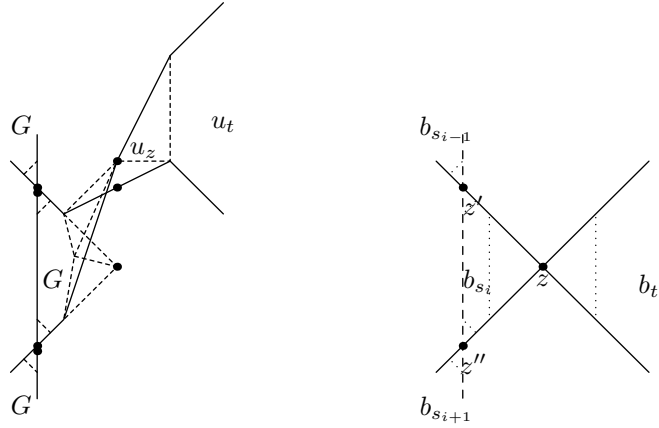}
\caption{For $z \in \C_{II}(s) \cap \C(t),$ $t \in \S(\D),$ and the band $u_z$ connecting $G$ and $u_{t}$ on the left and $\mathrm{Im}(d_2) \subset \R^2$ and $z', z'' \in \C(\D_2) \setminus \C(\D)$ around the crossing $z$ on the right.}\label{krr}
\end{figure} 

Now $M_2$ is isotopic to $M$ and thus satisfies (ii). By property (d) $M_2$ satisfies (iv). 
Towards proving (iii), let $\Omega_0, \Omega_1, \ldots, \Omega_{k} \subset \R^2,$ where $k=2 \#\C_{II}(s),$ be the collection of components of $p(G)\setminus p(\al \cup \be_2)$ and $s_i:=\partial_{\R^2}(\Omega_i)$ for every $i \in \{0, \ldots,k\}.$ Since $p(\be_2) \cap \mathrm{Im}(d)\subset p(\al)$ and $p(\al) \subset \mathrm{Im}(s),$ the collection $\S(\D_2)$ of Seiferts circles of $\D_2$ satisfies 
$$\S(\D_2)= (\S(\D) \setminus \{s\}) \cup \{s_i: i \in \{0, \ldots, k\}\}$$ and $b_{s_i}=\Omega_i.$
Now $s_i \notin \S_{II}(\D_2)$ for every $i \in \{0, \ldots, k\},$ since by property (a) we have $b_{s_i} \cap \mathrm{Im}(d_2) \subset (p(G) \setminus p(\al)) \cap \mathrm{Im}(d_2)=\es.$ Since $s \in \S_{II}(\D)$ and $s \notin \S(\D_2),$ we have $\S_{II}(\D_2)=\S_{II}(\D) \setminus \{s\}.$ Thus $M_2$ satisfies (iii).

It remains to show that $M_2$ satisfies condition (i). 
By property (b) the projection $p \colon \R^3 \ra \R^2$ induces $\D_2$ from $\mathrm{Bd}(M_2)$ and by property (c) the crossing area $w_i$ of $\D$ is a crossing area for $\D_2$ for every $i \in \C(\D).$ Thus $(w_i)_{i \in \C(\D_2)}$ is a pairwise disjoint collection of crossing areas and there exists a collection $(u'_i)_{i \in (\C(\D_2)\setminus \C(\D)) \cup (\S(\D_2)\setminus \S(\D))}$ of closed PL 2-manifolds with one boundary component and genus $0$ in $\mathrm{Cl}\big(G \setminus \big(\Cup_{i \in C_{II}(s)}w_i\big)\big)$
so that when we set $u'_i=u_i \cup (p^{-1}(w_i) \cap G)$ for every $i \in \C_{II}(s)$ and
$u'_i=u_i$ for every $i \in (\S(\D) \cap \S(\D_2)) \cup (\C(\D) \setminus \C_{II}(s))$ and $r_i=r_s$ for every $i \in \S(\D_2)\setminus \S(\D)$ the collection 
$$((u'_i)_{i \in \S(\D_2) \cup \C(\D_2)},(w_i)_{i \in \C(\D_2)},(r_i)_{i \in \S(\D_2)})$$ is a data of $M_2$ with respect to $\D_2.$ Thus $M_2$ is canonical with respect to $\D_2.$ Thus $M_2$ satisfies condition (i). 
\end{proof}

\begin{Cor}\label{4} Let $\D$ be a knot diagram and $M \subset \R^3$ a Seifert surface that is canonical with respect to $\D.$ Then there exists a knot diagram $\D'$ and a Seifert surface $M'$ so that the open Seifert disks in $\D'$ are pairwise disjoint, $M'$ is canonical with respect to $\D'$ and isotopic to $M$ and 
$$\#\C(\D')-\#\C(\D)=\sum_{i \in \S(\D)} 2\#\C_{II}(i).$$
\end{Cor} 

\begin{proof}By Lemma \ref{1} the collection of Seifert disks in a knot diagram $\D'$ is pairwise disjoint if $\S_{II}(\D')=\es.$ Thus the statement follows from Proposition \ref{3} by induction on $\#(\S_{II}(\D)).$
\end{proof}

This concludes the proof of Theorem \ref{tulos}. 



\section{Algorithm}

Let $\D=(d,f)$ be a knot diagram of a knot $[k].$ The integer $\G(\D)=(\#\C(\D)-\#\S(\D)+1)/2$ does not depend on the function $f.$ Thus the proof of Proposition \ref{3} yields an algorithm to construct a knot diagram $\D'$ of $[k]$ so that $S_{II}(\D')=\es$ and $\G(\D')=\G(\D).$

Suppose $\S_{II}(\D)\neq \es.$ Let $s \in \S_{II}(\D)$ and $[x_1,x_2] \subset d^{-1}(\mathrm{Im}(s)\setminus \C(\D)).$ By tracing $\mathrm{Im}(s) \setminus s([x_1,x_2])$ from $s(x_2)$ and crossing $\mathrm{Im}(s)$ twice close to every $i \in \C_{II}(s)$ on top of $\mathrm{Im}(s)$ to avoid crossing $\mathrm{Im}(d) \setminus \mathrm{Im}(s)$ we get a knot digram $\D':=(d',f')$ so that
  
\begin{itemize}
\item[(a)] $d \r (S^1 \setminus [x_1,x_2])=d' \r  (S^1 \setminus [x_1,x_2])$ and 
\begin{displaymath}\left\{ \begin{array}{ll}

&f'(x)=f(x) \text{ for } x \in d^{-1}(\C(\D)),\\
&f'(x)=1 \text{ for } x \in d'^{-1}((\C(\D')\setminus \C(\D)) \cap [x_1,x_2]\\  
&f'(x)=0 \text{ for } x \in d'^{-1}((\C(\D')\setminus \C(\D)) \setminus [x_1,x_2] ,
\end{array} \right.
\end{displaymath}
\item[(b)] $\S(\D)\setminus \S(\D')=\{s\}$ and $\#(\S(\D')\setminus \S(\D))=\#(\C(\D')\setminus \C(\D))+1$ and  
\item[(c)] $\S_{II}(\D')\subset \S_{II}(\D).$

\end{itemize}
Now $\D':=(d',f')$ is a knot diagram of $[k]$ satisfying $\#\S_{II}(\D') < \#\S_{II}(\D)$ and $\G(\D')=\G(\D).$ Thus the Seifert circuits $\S_{II}(\D)$ can be removed from $\D$ one by one. 


\bibliography{viite}{}

\bibliographystyle{amsplain}

\end{document}